\documentclass[11pt]{amsart}
\usepackage[english]{babel}
\usepackage[utf8]{inputenc}
\usepackage{graphicx}
\usepackage{xcolor}
\usepackage{longtable}
\usepackage{float}
\usepackage[all]{xy}
\usepackage{amssymb,amscd}
\usepackage{tikz}
\usepackage{csquotes}
\usepackage{imakeidx}
\usepackage{appendix}
\usepackage{tikz-cd}
\usepackage{enumitem}
\usepackage[colorlinks=true,linkcolor=blue,citecolor=blue,urlcolor=magenta,pagebackref=true]{hyperref}
\usepackage{cleveref}
\usepackage{calc}
\usepackage{mathtools}
\usepackage{comment}

\numberwithin{equation}{section}

\newtheorem{theorem}{Theorem}[section]
\newtheorem{proposition}[theorem]{Proposition}
\newtheorem{lemma}[theorem]{Lemma}
\newtheorem{corollary}[theorem]{Corollary}
\theoremstyle{remark}
\newtheorem{remark}[theorem]{Remark}

\newcommand{\A}{\mathbb A}
\newcommand{\PP}{\mathbb P}
\newcommand{\OO}{\mathcal O}
\newcommand{\bir}{\mathrel{\sim_{\mathrm{bir}}}}

\DeclareMathOperator{\Disc}{Disc}

\DeclareMathOperator{\Nm}{N}
\DeclareMathOperator{\Spec}{Spec}

\title{Noether's problem for $A_6$ and $A_7$}
\author{Federico Scavia}
\address{CNRS, Institut Galil\'ee, Universit\'e Sorbonne Paris Nord, 99 avenue Jean-Baptiste Cl\'ement, 93430 Villetaneuse, France}
\email{scavia@math.univ-paris13.fr}
\date{September 12, 2026}

\begin{document}
	\begin{abstract}
		We prove that $B_kA_6$ is stably rational over every field $k$ in which $2$ and $-3$ are nonzero squares. Combining this with a theorem of Plans, we deduce that $B_kA_7$ is stably rational over every field $k$ of characteristic zero in which $2$ and $-3$ are squares.
	\end{abstract}
	
	\maketitle
	
	\section{Introduction}
	
	The rationality problem for fields of invariants was first posed by Noether
	\cite{Noe13,Noe18}. In its classical form, one lets a finite group $G$
	act by left translation on a set of variables $\{x_g\}_{g\in G}$ and asks
	whether the field
	\[
	k(G)\coloneqq k(x_g:g\in G)^G
	\]
	is rational over the ground field $k$. Today it is also natural to consider the stable birational form of Noether's problem. Recall that, by the no-name lemma, if $V$ and $V'$ are generically free finite-dimensional linear representations of $G$, then the quotients $V/G$ and $V'/G$ are stably birational; see for example \cite[Corollary~3.9]{CTS07}. We denote this common stable birational type by $B_kG$. The stable birational form of Noether's problem asks whether, for some (equivalently, every) faithful linear representation $V$ of $G$, the field $k(V)^G$ is stably rational over $k$.
	
	Noether's problem has a long history and is closely related to generic $G$-extensions and to the inverse Galois problem; see, for example, \cite{Noe18,CTS07}. The answer is negative in general: the first counterexamples to Noether's problem were obtained independently by Swan \cite{Swa69} and Vos\-kre\-senski\u{\i} \cite{Vos70}, who showed that $\mathbb Q(C_{47})$ is not rational over $\mathbb Q$. In fact, $\mathbb Q(C_{47})$ is not even stably rational, by a later result of Voskresenski\u{\i} \cite[Proposition~12]{Vos73}.
    In contrast, by a theorem of Saltman \cite[Theorem 2.1]{Sal82}, $B_{\mathbb Q}C_p$ is retract rational for every prime $p$. A smaller counterexample is $B_{\mathbb Q}C_8$, considered by Endo--Miyata \cite{EM73}, Voskresenski\u{\i} \cite{Vos73}, and Lenstra \cite{Len74}, which is not even retract rational. Saltman \cite{Sal82,Sal84b} gave a particularly simple explanation of this failure using Wang's counterexample to Grunwald's theorem \cite{Wan48}; see \cite[\S\S7--8]{CTS87} for a discussion of the relation between the Grunwald--Wang problem, weak approximation, and retract rationality. 
    
    The first counterexamples to Noether's problem over algebraically closed fields are due to Saltman \cite{Sal84}: for every prime $p$, he constructed a nonabelian $p$-group $G$ of order $p^9$ such that $k(G)$ is not retract
	rational, for $k$ algebraically closed of characteristic different from $p$. His obstruction is the nonvanishing of the unramified Brauer group. Bogomolov \cite{Bog88} subsequently gave a group-theoretic description of this invariant, now called the Bogomolov multiplier of $G$, and used this description to construct smaller counterexamples, with $G$ a $p$-group of order $p^6$.
	
	The natural permutation representations of the alternating groups give one of the most classical cases of Noether's problem. For the symmetric group $S_n$, one has
	\[k(x_1,\ldots,x_n)^{S_n}=k(e_1,\ldots,e_n),\]
	where $x_1,\dots,x_n$ are independent variables permuted by $S_n$, and where $e_1,\dots,e_n$ are the elementary symmetric functions. 
	For the alternating group $A_n$, when $\mathrm{char}(k)\neq 2$, one has
	\[k(x_1,\ldots,x_n)^{A_n}=k(e_1,\ldots,e_n)(\sqrt{\Delta}),\]
	where $\Delta=\Delta(e_1,\dots,e_n)\in k[e_1,\dots,e_n]$ is the discriminant of the polynomial
	\[\prod_{i=1}^n(T-x_i)=T^n-e_1T^{n-1}+\dots+(-1)^ne_n.\]
	Despite this simple-looking description, the rationality problem for $A_n$ is difficult already for small $n$.
    For $n=3$ and $4$, rationality was already known classically. More generally, when $\mathrm{char}(k)=0$, Noether \cite{Noe18} and Seidelmann \cite{Sei16} proved rationality of $k(x_1,\dots,x_n)^G$ for the transitive subgroups $G$ of $S_n$, where $n=3,4$. Kuyk \cite{Kuy60} later extended this to fields of characteristic different from $2$; see also \cite[p. 38]{Kuy64}. The next case was settled by Maeda \cite{Mae89}, who proved that
\[
	\text{$k(x_1,\ldots,x_5)^{A_5}$ is $k$-rational for every field $k$.}
\]
    Plans \cite{Pla09} later proved that, for every odd positive integer $n$, the invariant field of the regular representation $\mathbb Q(A_n)$ is rational over $\mathbb Q(A_{n-1})$; it follows that $k(A_n)/k(A_{n-1})$ is rational for every odd $n$ and every field $k$ of characteristic zero.
	
	All known results for transitive permutation representations of degree $n=6,7$ exclude $A_n$. Kang--Wang \cite{KW14} proved that $\mathbb C(x_1,\ldots,x_7)^G$ is rational for every transitive subgroup $G\subset S_7$ other than possibly $A_7$. In degree six, Kang--Wang--Zhou \cite[Theorem~1.2]{KWZ15} proved that
	$k(x_1,\ldots,x_6)^G$ is rational for every subgroup $G\subset S_6$, with the possible exceptions of groups isomorphic to $PSL_2(\mathbb F_5)$, $PGL_2(\mathbb F_5)$, and $A_6$;  they proved stable rationality over an arbitrary field for the first two groups, but the case of $A_6$ remained open.
	
	Our main result is the following.
	
	\begin{theorem}\label{thm:main}
		Let $k$ be a field with $2,-3\in k^{\times2}$, let $x_1,\dots,x_6$ and $t$ be algebraically independent variables over $k$, and let $A_6$ act on $k(x_1,\ldots,x_6)$ by permutation of the variables. Then $k(x_1,\ldots,x_6)^{A_6}(t)$ is rational over $k$. In particular, $B_kA_6$ is stably rational over $k$.
	\end{theorem}
	
	In particular, the theorem holds over every algebraically closed field of characteristic not $2$ or $3$. Even though $5$ divides the order of $A_6$, the argument also works in characteristic $5$.
	
	Combining Theorem~\ref{thm:main} with Plans' theorem \cite{Pla09} gives the following consequence.
	
	\begin{corollary}\label{cor:A7}
		Let $k$ be a field of characteristic zero such that
		$2,-3\in k^{\times2}$. Then $B_kA_7$ is stably $k$-rational.
	\end{corollary}
	
	By a theorem of Merkurjev \cite[Theorem~6.1]{Mer20}, $B_kA_n$ is $p$-retract rational for every prime $p$. On the other hand, before the present work it was not even known whether $B_kA_6$ or $B_kA_7$ were retract rational. 
	
	Let $k$ be a number field such that $2,-3\in k^{\times2}$. Theorem~\ref{thm:main} implies that $SL_n/A_6$ is stably rational for any $k$-group embedding $A_6\hookrightarrow SL_n$. We obtain the following arithmetic consequences; see \cite[\S2.3]{BN26} and \cite[Proposition~2.4]{DLAN17}.
	
	\begin{corollary}\label{cor:BM}
		Let $k$ be a number field such that $2,-3\in k^{\times2}$.
		Then $A_6$ satisfies the Brauer--Manin condition for weak
		approximation. In particular, $A_6$ has the tame approximation
		property over $k$: the Grunwald problem for $A_6$ at finite sets of
		places away from those dividing $2,3,5$ has an affirmative answer.
	\end{corollary}

	We outline the proof of Theorem~\ref{thm:main}. Let
	\[
	W=\{(x_1,\ldots,x_6)\in k^6:\textstyle\sum_i x_i=0\},
	\qquad
	Q=\{[x]\in\PP_k(W):\textstyle\sum_i x_i^2=0\}.
	\]
	The variety $Q\subset\PP^4_k$ is a smooth quadric threefold which is preserved by the permutation action of $S_6$. Over $\mathbb C$, this is the $A_6$-quadric considered in
	\cite[Example~1.14]{CS14}.
	
	We will prove Theorem~\ref{thm:main} as a consequence of the following theorem.
	
	\begin{theorem}\label{thm:quadric}
		Let $k$ be a field of characteristic not equal to $2$ or $3$. The field extension $k(Q)^{A_6}/k$ is rational.
	\end{theorem}
	
	Assume for the moment that Theorem~\ref{thm:quadric} has been proved, and consider
	\[
	M\coloneqq (\PP(W)\times Q)/A_6,
	\]
	where the quotient is taken with respect to the diagonal action. The two projections induce rational maps
	\[
	M\dashrightarrow Q/A_6,
	\qquad
	M\dashrightarrow \PP(W)/A_6.
	\]
	The generic fiber of the first map is a projective space over $k(Q)^{A_6}$, and the generic fiber of the second map is a twist $\mathcal{Q}$ of the quadric $Q$ over $k(\PP_k(W))^{A_6}$. We show that the smooth quadric $\mathcal{Q}$ acquires a rational point over the odd-degree extension $k(\PP_k(W))^{P}/k(\PP_k(W))^{A_6}$, where $P$ is a Sylow $2$-subgroup of $A_6$; here we use the assumption $2,-3\in k^{\times2}$ to construct a $P$-equivariant morphism $\PP^1_k\to Q$, where $P$ acts linearly on $\PP^1_k$. Springer's theorem then allows us to descend the point to $k(\PP_k(W))^{A_6}$. Thus $\mathcal{Q}$ is rational over $k(\PP_k(W))^{A_6}$, and comparison of the two projections yields Theorem~\ref{thm:main}.
	
	We now discuss the proof of Theorem~\ref{thm:quadric}. Since $\sum_i x_i^2=e_1^2-2e_2$, the affine cone over $Q$ is cut out by $e_1=e_2=0$. Therefore, the quotient $Q/S_6$ is birationally equivalent to the weighted projective space $\PP(3,4,5,6)$, and the intermediate quotient $Q/A_6$ is the discriminant double cover. In Proposition~\ref{prop:coefficientchart}, we choose independent parameters $p,q,\lambda$ over $k$ so that
	\[
	k(Q)^{A_6}
	=
	k(p,q,\lambda)
	\left(
	\sqrt{
		\Disc_T(\lambda T^6+pT^3+T^2+T+q)}
	\right).
	\]
	(See Remark~\ref{rem:why-pql} for motivation for this choice of the parameters $p,q,\lambda$.) 
	Letting $K\coloneqq k(q)$, we will prove the stronger assertion that $k(Q)^{A_6}$ is rational over $K$. After the change of variables $U=1/T$, Proposition~\ref{prop:surfacepencil} shows that
	\[k(Q)^{A_6}=K(p,\lambda)\left(\sqrt{\Disc_U(H_p(U)+\lambda)}\right), \quad H_p(U)=qU^6+U^5+U^4+pU^3.\]
	Let
	\[
	Y\longrightarrow \PP^1_{p,K}\times_K\PP^1_{\lambda,K}
	\]
	be a birational model of this double cover. If either projection exhibited
	$Y$ as a split conic bundle over a rational curve, then $Y$ would be
	$K$-rational. However, the generic fibers of the two projections are not conics: over the $p$-line one obtains a double cover of $\PP^1_\lambda$ branched at $4$ points, hence a curve of genus $1$, while the generic fiber over the $\lambda$-line is a double cover of
	$\PP^1_p$ branched at $6$ points, hence a curve of genus $2$. We will prove Theorem~\ref{thm:quadric} by making a suitable change of coordinates \[K(p,\lambda)=K(p,u)\] after which the resulting birational model has the structure of a split conic bundle over the $u$-line, and hence is $K$-rational. Such fiberwise changes of coordinates are classical;
	Serre \cite[p.~76, \S1.1(c)]{Ser10} refers to them as de Jonqui\`eres transformations.
	
	Let $F\coloneqq K(p)$. Over the generic point of the $p$-line, the fiber of $Y$ is the degree-$2$ cover $Y_F\to\PP^1_{\lambda,F}$ corresponding to the quadratic extension
	\[
	F(\lambda)\subset
	F(\lambda)\left(
	\sqrt{\Disc_U(H_p(U)+\lambda)}
	\right).
	\]
	Consider also the degree-$6$ morphism
	\[
	\varphi_F\colon\PP^1_{u,F}\longrightarrow\PP^1_{\lambda,F},
	\qquad
	u\longmapsto-H_p(u).
	\]
	Since 
	\[
	H_p'(u)=u^2J_p(u),
	\qquad
	J_p(u)\coloneqq 6qu^3+5u^2+4u+3p,
	\]
	the branch points of $Y_F\to\PP^1_{\lambda,F}$ are the three values
	\[
	\lambda=\varphi_F(u)=-H_p(u),
	\qquad
	J_p(u)=0,
	\]
	together with $\lambda=\infty$. (The ramification at $u=0$ contributes the square factor $\lambda^2$ to the discriminant and therefore does not give a branch point of the quadratic
	cover.) 
	
	Let $C\subset \PP^1_{p,K}\times_K\PP^1_{u,K}$ be the closure of the locus $J_p(u)=0$. Since the equation $J_p(u)=0$ is linear in $p$, the curve $C$ is the graph of a degree-$3$ morphism
	$\PP^1_{u,K}\to\PP^1_{p,K}$. Let $C_F\coloneqq C\times_{\PP^1_{p,K}}\Spec F\subset\PP^1_{u,F}$. Then $C_F$ is finite \'etale of degree $3$, and its three geometric points have distinct images	under $\varphi_F$; see Lemmas~\ref{lem:ramificationcurve} and \ref{lem:ramificationimages}. Hence there is a unique projective isomorphism
	\[
	\Phi\colon\PP^1_{u,F}\xlongrightarrow{\sim}\PP^1_{\lambda,F}
	\]
	whose restriction to $C_F$ agrees with $\varphi_F$.
	
	We pull back the extension $K(Y)/F(\lambda)$ along the induced isomorphism $\Phi^*\colon F(\lambda)\xrightarrow{\sim}F(u)$. Taking the normalization of $\PP^1_{p,K}\times_K\PP^1_{u,K}$ in the new quadratic extension gives a double cover
	\[
	\pi\colon Z\longrightarrow \PP^1_{p,K}\times_K\PP^1_{u,K}.
	\]
	Over the generic point $\Spec F$ of $\PP^1_{p,K}$, the cover $\pi$ is branched precisely along
	$C_F$ and the point $\Phi^{-1}(\infty)$. If $j(u)=0$ defines $C_F$ and $D(u)=0$ defines $\Phi^{-1}(\infty)$, then
	\[K(Z)=F(u)\left(\sqrt{w(p)\,j(u)D(u)}\right)
	\]
	for some $w(p)\in F^\times$, well-defined modulo squares. Letting $E\subset \PP^1_{p,K}\times_K\PP^1_{u,K}$ be the closure of the locus $D(u)=0$, we see that $C$ and $E$ have bidegrees $(1,3)$ and $(1,1)$, respectively, and have no common irreducible component.
	
	Since $C+E$ has bidegree $(2,4)=2(1,2)$, there exist double covers of $\PP^1_{p,K}\times_K\PP^1_{u,K}$ with branch divisor $C+E$. However, it is not clear why $\pi$ should be one of them: indeed,  $\pi$ may also branch where $w(p)$ has odd valuation. It turns out that the square class of $w(p)$ is in fact constant, that is, it comes from $K$. More precisely, Proposition~\ref{prop:exactcover} shows that $w(p)=q$ in $F^{\times}/F^{\times 2}$. Thus
	\[K(Z)=K(p,u)\left(\sqrt{qj(u)D(u)}\right),\]
	and the branch divisor of $\pi$ is $C+E$. The proof of Proposition~\ref{prop:exactcover} is by explicit algebraic manipulations; I do not know a geometric explanation for why the branch divisor of $\pi$ is exactly $C+E$.
	
	Consider the projection $Z\to \PP^1_{u,K}$. Since $C$ and $E$ have bidegrees $(1,3)$ and $(1,1)$, respectively, the generic fiber is a degree-$2$ cover of $\PP^1_{p,K(u)}$ branched at exactly two $K(u)$-rational points. These points are distinct because $C$ and $E$ have no common component. Thus the generic fiber is a smooth conic over $K(u)$. Moreover, the ramification point lying above $C$ is $K(u)$-rational, since $C$ has degree one over the $u$-line. Hence the generic conic has a rational point and is therefore $K(u)$-rational. It follows that $Z$ is $K$-rational. Since $Z$ is birational to $Y$, Proposition~\ref{prop:surfacepencil} now gives $k(Q)^{A_6}=K(Y)=K(Z)$, which proves Theorem~\ref{thm:quadric}.

	\section*{Acknowledgements}
	
	I thank my former advisor Zinovy Reichstein for introducing me to this problem when I was a PhD student.
	
	I have used a publicly available large language model and a computer algebra system to assist in determining the correct formulation of Lemma~\ref{lem:poledegrees} and Proposition~\ref{prop:exactcover}.

	\section{Proof of Theorem~\ref{thm:quadric}}
	
	\subsection{The sextic pencil}\label{sec:coefficients}
	
	Let $k$ be a field of characteristic not equal to $2$ or $3$. As in the
	Introduction, let
	\[
	W=\{(x_1,\ldots,x_6)\in k^6:\textstyle\sum_i x_i=0\},
	\qquad
	Q=\{[x]\in\PP_k(W):\textstyle\sum_i x_i^2=0\}.
	\]
	
	\begin{proposition}\label{prop:coefficientchart}
		There exist algebraically independent elements
		$p,q,\lambda$ of $k(Q)^{S_6}$ such that
		\[
		k(Q)^{S_6}=k(p,q,\lambda)
		\]
		and
		\[
		k(Q)^{A_6}
		=
		k(p,q,\lambda)
		\left(
		\sqrt{
			\Disc_T(\lambda T^6+pT^3+T^2+T+q)}
		\right).
		\]
	\end{proposition}
	
	See Remark~\ref{rem:why-pql} for motivation for the choice of $p,q,\lambda$.
	
	\begin{proof}
		Let $\widehat Q\subset W$ be the affine cone over $Q$, and let
		$e_1,\dots,e_6$ be the elementary symmetric functions in
		$x_1,\dots,x_6$. Since $\sum_i x_i^2=e_1^2-2e_2$, the equations of $\widehat Q$ are $e_1=e_2=0$. Thus the polynomial with roots $x_1,\dots,x_6$ is 
		\[
		P(T)=T^6-e_3T^3+e_4T^2-e_5T+e_6.
		\]
		
		Set
		\[
		\lambda\coloneqq\frac{e_5^4}{e_4^5},
		\qquad
		p\coloneqq\frac{e_3e_5}{e_4^2},
		\qquad
		q\coloneqq\frac{e_4e_6}{e_5^2},
		\qquad
		s\coloneqq-\frac{e_5}{e_4}.
		\]
		Then
		\[
		e_3=-\frac{ps^3}{\lambda},
		\qquad
		e_4=\frac{s^4}{\lambda},
		\qquad
		e_5=-\frac{s^5}{\lambda},
		\qquad
		e_6=\frac{qs^6}{\lambda},
		\]
		so
		\[
		k(e_3,e_4,e_5,e_6)=k(p,q,\lambda,s).
		\]
		The scalar action on $\widehat Q$ sends $e_i$ to $t^ie_i$. Thus it
		fixes $p,q,\lambda$ and sends $s$ to $ts$. Since $k(\widehat Q)^{S_6}=k(e_3,e_4,e_5,e_6)$,
		taking invariants under the scalar action yields $k(Q)^{S_6}=k(p,q,\lambda)$.	
		
		Define $\varepsilon\coloneqq\prod_{i<j}(x_i-x_j)$. Then
		\[k(\widehat Q)^{A_6}=k(e_3,e_4,e_5,e_6)(\varepsilon),\qquad \varepsilon^2=\Disc_T(P)\in k(e_3,e_4,e_5,e_6);\]
		see for example \cite[Proposition~1.3.1 and Corollary~1.3.2]{Smi95}. Since $\varepsilon$ has weight $15$, the element $\eta\coloneqq \varepsilon/s^{15}$ is invariant under the scalar action, and hence $k(Q)^{A_6}=k(p,q,\lambda)(\eta)$.
		
		Finally, recall that for a polynomial of degree $6$, the discriminant satisfies
		\[
		\Disc(F(sV))=s^{30}\Disc(F),
		\qquad
		\Disc(cF)=c^{10}\Disc(F).
		\] 
		Therefore, since $P(sV)=(s^6/\lambda)(\lambda V^6+pV^3+V^2+V+q)$, 
		we have 
		\[
		s^{30}\Disc_T(P)
		=
		\frac{s^{60}}{\lambda^{10}}
		\Disc_V(\lambda V^6+pV^3+V^2+V+q),
		\]
		that is,
		\[
		\eta^2
		=
		\lambda^{-10}
		\Disc_V(\lambda V^6+pV^3+V^2+V+q),
		\]
		from which the conclusion follows.
	\end{proof}

	Set $K\coloneqq k(q)$, and define
	\[
	H_p(U)\coloneqq qU^6+U^5+U^4+pU^3\in K[p][U].
	\]
	Let $Y$ be the normalization of
	$\PP^1_{p,K}\times_K\PP^1_{\lambda,K}$ in the quadratic extension
	\[
	K(p,\lambda)
	\subset
	K(p,\lambda)\left(
	\sqrt{
		\Disc_T(\lambda T^6+pT^3+T^2+T+q)}
	\right)
	\]
	of Proposition~\ref{prop:coefficientchart}. We have a finite degree-two morphism
	\[
	Y\longrightarrow\PP^1_{p,K}\times_K\PP^1_{\lambda,K}.
	\]
	
	\begin{proposition}\label{prop:surfacepencil}
		We have
		\[
		k(Q)^{A_6}=K(Y)=
		K(p,\lambda)
		\left(
		\sqrt{\Disc_U(H_p(U)+\lambda)}
		\right).
		\]
	\end{proposition}
	
	\begin{proof}
		By Proposition~\ref{prop:coefficientchart} and the definition of $Y$, we have
		\[
		k(Q)^{A_6}=K(Y)
		=
		K(p,\lambda)
		\left(
		\sqrt{
			\Disc_T(\lambda T^6+pT^3+T^2+T+q)}
		\right).
		\]
		Replacing $T$ by $U^{-1}$ and multiplying by $U^6$ gives
		\[
		H_p(U)+\lambda=qU^6+U^5+U^4+pU^3+\lambda.
		\]
		It now follows from \cite[p.~404, (1.26)]{GKZ94} that \[\Disc_T(\lambda T^6+pT^3+T^2+T+q)=\Disc_U(H_p(U)+\lambda),\] completing the proof.
	\end{proof}

	\subsection{The ramification curve}
	\label{sec:ramification}
	
	Consider the morphism
	\[
	\varphi\colon\A^1_{p,K}\times_K\PP^1_{u,K}\longrightarrow
	\A^1_{p,K}\times_K\PP^1_{\lambda,K},\qquad
	(p,u)\longmapsto(p,-H_p(u)).
	\]
	For $p_0\in\A^1_{p,K}$, write $\varphi_{p_0}$ for the induced degree-six
	morphism
	\[\varphi_{p_0}\colon\PP^1_{u,K}\longrightarrow\PP^1_{\lambda,K},\qquad u\longmapsto-H_{p_0}(u).\]
	Differentiation gives
	\begin{equation}\label{eq:ramificationcubic}
		H_p'(U)=U^2J_p(U),\qquad J_p(U)=6qU^3+5U^2+4U+3p.
	\end{equation}
	Thus, for the fiber over $p_0$, the ramification points other than $u=0$ and $u=\infty$ are the roots of $J_{p_0}(U)=6qU^3+5U^2+4U+3p_0$.
	
	\begin{lemma}\label{lem:ramificationcurve}
		Define 
		\[C\coloneqq \overline{\{(p,u): J_p(u)=0\}}\subset \PP^1_{p,K}\times_K\PP^1_{u,K}.\]
		\begin{enumerate}
			\item Consider the morphism \[f\colon\PP^1_{u,K}\longrightarrow\PP^1_{p,K},\qquad
			f(u)=-2qu^3-\frac53u^2-\frac43u.\]
			Then $f$ is separable of degree $3$, and $C=\Gamma_f$ is the graph of $f$. In particular, the divisor $C\subset \PP^1_{p,K}\times_K\PP^1_{u,K}$ is irreducible and has bidegree $(1,3)$.
			\item The restriction of $\varphi$ to $C$ induces a rational map $h\colon \PP^1_{u,K}\dashrightarrow\PP^1_{\lambda,K}$,
			where we identify $C\simeq\PP^1_{u,K}$ via
			$\operatorname{pr}_u|_C$. This rational map extends to a morphism
			\[
			h\colon \PP^1_{u,K}\longrightarrow\PP^1_{\lambda,K},
			\qquad
			h(u)=qu^6+\frac{2}{3}u^5+\frac{1}{3}u^4.
			\]
		\end{enumerate}
	\end{lemma}
	
	\begin{proof}
		(1) The equation $J_p(u)=0$ is linear in $p$. Solving it for $p$
		gives $p=f(u)$, so $C$ is precisely the graph of $f$ on the affine
		chart. Both polynomial maps extend to morphisms of
		projective lines by sending infinity to infinity. The polynomial $f$ has degree three. It is separable because $\mathrm{char}(k)$ is different from $2$ or $3$.
		
		(2) It suffices to substitute $p=f(u)$ into $-H_p(u)$:
		\[
		h(u)=-H_{f(u)}(u)
		=qu^6+\frac23u^5+\frac{1}{3}u^4.\qedhere
		\]
	\end{proof}

	\begin{remark}\label{rem:why-pql}
		The choice of the parameters $p,q,\lambda$ in Proposition~\ref{prop:coefficientchart} is motivated by Lemma~\ref{lem:ramificationcurve}. Indeed, in the polynomial
		\[
		H_p(U)+\lambda
		=
		qU^6+U^5+U^4+pU^3+\lambda,
		\]
		the parameter $\lambda$ appears as the constant term and nowhere else, and hence plays the role of the value parameter for the map $u\mapsto-H_p(u)$. More
		importantly, since $p$ occurs in the lowest nonconstant term, the equation for the nonzero ramification points is linear in $p$. Solving for $p$ therefore gives the polynomial cubic
		\[
		p=-2qu^3-\frac53u^2-\frac{4}{3}u,
		\]
		which is precisely why the ramification curve $C$ is the graph of a
		degree-$3$ morphism $\PP^1_{u,K}\to\PP^1_{p,K}$.
	\end{remark}
	
	\begin{lemma}\label{lem:ramificationimages}
		View the morphisms $f\colon \PP^1_{u,K}\to\PP^1_{p,K}$ and $h\colon \PP^1_{u,K}\to\PP^1_{\lambda,K}$ as rational functions $f,h\in K(u)$ on $\PP^1_{u,K}$ by pulling back the affine coordinates $p$ and $\lambda$. 
		
		\begin{enumerate}
			\item The rational function $h-f^2/(4q)\in K(u)$ has a pole of order $5$ at infinity and no poles on $\A^1_{u,K}$.
			\item We have $h\notin K(f)$ and $K(f,h)=K(u)$.
			\item Set $F\coloneqq K(p)$, and let $C_F\coloneqq C\times_{\PP^1_{p,K}}\Spec F
			\subset\PP^1_{u,F}$. Then $C_F$ is a degree-$3$ finite \'etale $F$-scheme, and its three
			geometric points have distinct images under $h$.
		\end{enumerate}
	\end{lemma}

	\begin{proof}
		(1) One computes
		\[
		\frac{f(u)^2}{4q}
		=
		qu^6+\frac53u^5+\cdots.
		\]
		Hence
		\[
		h(u)-\frac{f(u)^2}{4q}=-u^5+\cdots
		\]
		is a polynomial of degree $5$. This proves (1).
		
		(2) Suppose that $h\in K(f)$. Then $h-f^2/(4q)=g(f)$ for some $g\in K(p)$. If $g$ had a pole at a point $a\in\A^1_{p,K}$, then $g(f)$ would have a pole at every point of $f^{-1}(a)$. Since $f^{-1}(\infty)=\{\infty\}$, the fiber $f^{-1}(a)$ is contained in $\A^1_{u,K}$, contradicting (1). Thus $g$ has no poles on $\A^1_{p,K}$, so $g\in K[p]$.
		
		Since $f$ has a pole of order $3$ at infinity, every nonconstant polynomial in $f$ has pole order divisible by $3$. By (1), the function $h-f^2/(4q)$ has pole order $5$, a contradiction. Therefore $h\notin K(f)$. Since $[K(u):K(f)]=3$, it follows that $K(f,h)=K(u)$.
		
		(3) Identify $F=K(p)$ with the subfield $K(f)\subset K(u)$ via the pullback $f^*\colon F\hookrightarrow K(u)$, which sends $p$ to $f(u)$. Since $C$ is the graph of $f$, its generic fiber $C_F$ is $\Spec K(u)$. By (2), the extension $K(u)/F$ is separable of degree $3$, so $C_F$ is finite \'etale of degree $3$ over $F$. Moreover, $K(u)=F(h)$, so the values of $h$ at the three geometric points of $C_F$ are the three conjugates of $h$ over $F$. These are distinct because $K(u)/F$ is separable.
	\end{proof}

	The following description of $h$ will be used to prove Proposition~\ref{prop:fourthdivisor} below.
	
	\begin{lemma}\label{lem:poledegrees}
		Using the inclusion $K[p]\hookrightarrow K[u]$ given by $p\mapsto f(u)$, there is a unique expression
		\[
		h(u)=\alpha(p)u^2+\beta(p)u+\gamma(p),
		\]
		where $\alpha,\beta,\gamma\in K[p]$ satisfy $\deg\alpha=1$, $\deg\beta\leq 1$, and $\deg\gamma\leq 2$.
	\end{lemma}
	
	Here we allow $\beta=0$ and $\gamma=0$.
	
	\begin{proof}
		We have \[
		u^3
		=
		-\frac{p}{2q}
		-\frac{5}{6q}u^2
		-\frac{2}{3q}u.
		\]
		Thus every power $u^n$, $n\ge3$, can be reduced to a
		$K[p]$-linear combination of $1,u,u^2$. The reduction is
		unique, since $1,u,u^2$ are linearly independent over $K(p)$
		in the degree-three extension $K(u)/K(p)$. Applying this to $h(u)$
		gives 
		\[
		h(u)=\alpha(p)u^2+\beta(p)u+\gamma(p),
		\] 
		where $\alpha,\beta,\gamma\in K[p]$ are uniquely determined.
		
		For any nonzero $a(p)\in K[p]$, the function $a(f(u))$ has pole
		order $3\deg a$ at infinity. The nonzero summands on the right have poles of order $3\deg\gamma$, $3\deg\beta+1$ and $3\deg\alpha+2$. These numbers lie in different congruence classes modulo three.
		In particular, the largest pole cannot cancel with a pole from another
		summand. Since $h$ has pole order $6$, each of these orders is at
		most $6$. Thus $\deg\gamma\le2$, $\deg\beta\le1$ and $\deg\alpha\le 1$.
		
		Since $p=f(u)$ on $C$, we have
		\[
		h(u)-\frac{f(u)^2}{4q}
		=
		\left(\gamma(f(u))-\frac{f(u)^2}{4q}\right)
		+\beta(f(u))u+\alpha(f(u))u^2.
		\]
		By Lemma~\ref{lem:ramificationimages}(1), the left-hand side has a pole
		of order exactly $5$ at $u=\infty$.
		
		The first summand is a polynomial in $f(u)$, and hence, if nonzero, has pole order divisible by $3$ at infinity. The second summand has pole order at most $3\deg\beta+1\le4$, while the third has pole order at most $5$. Since these three pole orders are different modulo
		$3$, no cancellation is possible in the sum. By Lemma~\ref{lem:ramificationimages}(1), their sum has pole order exactly $5$. It follows that the third summand has pole order $5$, so $3\deg\alpha+2=5$, that is, $\deg\alpha=1$.
	\end{proof}

	\subsection{The computation of the change of coordinates \texorpdfstring{$\Phi$}{Phi}}
	\label{sec:coordinatechange}
	
	We now construct a projective isomorphism $\Phi\colon\PP^1_{u,F}\xrightarrow{\sim}\PP^1_{\lambda,F}$ which pulls back the degree-$3$ divisor of finite branch points of $Y_F\to\PP^1_{\lambda,F}$ to $C_F\subset\PP^1_{u,F}$.
	
	As in Lemma~\ref{lem:ramificationimages}(3), let 
	\[F=K(p)=K(\PP^1_{p,K}),\qquad C_F= C\times_{\PP^1_{p,K}}\Spec F
	\subset\PP^1_{u,F}.\]
	
	\begin{lemma}\label{lem:fiberiso}
		There is a unique projective isomorphism $\Phi\colon \PP^1_{u,F}\xrightarrow{\sim}\PP^1_{\lambda,F}$
		whose restriction to $C_F\subset\PP^1_{u,F}$ agrees with $h$.
	\end{lemma}
	
	\begin{proof}
		By Lemmas~\ref{lem:ramificationcurve}(1) and~\ref{lem:ramificationimages}, the
		generic fiber $C_F$ and its image under $h$ are isomorphic finite
		\'etale schemes of degree three over $F$. Over a separable
		closure of $F$, there is a unique projective isomorphism sending each of the three distinct points of the first triple to its prescribed image in
		the second. By Galois descent, it must be defined over $F$.
	\end{proof}

	By Lemma~\ref{lem:ramificationcurve}, we may 
	write $L=F[T]/(j(T))$, where
	\begin{equation}\label{eq:cubicfield}
		j(T)=\frac{J_p(T)}{6q}=T^3+bT^2+cT+e,
		\quad
		b=\frac5{6q},\quad c=\frac2{3q},\quad e=\frac p{2q}.
	\end{equation}
	Let $\theta\in L$ be the image of $T$, so that $L=F(\theta)$, and define $r\coloneqq h(\theta)\in L$. By Lemma~\ref{lem:poledegrees}, we have
	\begin{equation}\label{eq:defin-r}
		r=h(\theta)=\alpha(p)\theta^2+\beta(p)\theta+\gamma(p)\quad \text{in $L$,}
	\end{equation}
	where $\alpha,\beta,\gamma\in K[p]$ satisfy $\deg\alpha=1$, $\deg\beta\le1$, and $\deg\gamma\le2$. Geometrically, $\theta$ represents the generic point of the degree-$3$
	ramification subscheme of $\varphi_F\colon \PP^1_{u,F}\to\PP^1_{\lambda,F}$, while $r=h(\theta)$ is the corresponding generic branch value.

	We now write $\Phi$ as a fraction $\Phi=N/D$, where $N,D\in F[u]$ satisfy $\deg_u(D)=1$ and $\deg_u(N)\leq 1$. This will allow us to determine the pullback of the fourth branch point $\lambda=\infty$: it is the divisor $D=0$, whose closure has bidegree $(1,1)$; see Proposition~\ref{prop:fourthdivisor}(4) below. The same formulas for $N$ and $D$ will be used in Proposition~\ref{prop:exactcover} to compute the quadratic extension of $F(u)$ obtained from $K(Y)/F(\lambda)$ by pullback along $\Phi^*\colon F(\lambda)\xrightarrow{\sim}F(u)$.
	
	\begin{proposition}\label{prop:fourthdivisor}
		With the notation of \eqref{eq:cubicfield} and \eqref{eq:defin-r}, define the following elements of $K[p,u]\subset F[u]$: 
		\[
		R(u)\coloneqq \alpha u^2+\beta u+\gamma,\quad
		D(u)\coloneqq \alpha u+b\alpha-\beta,\quad
		N(u)\coloneqq D(u)R(u)-\alpha^2j(u).
		\]
		Observe that $R(\theta)=r$. Then the following statements hold.
		
		\begin{enumerate}
			\item We have $\deg_u(D)=1$, $\deg_u(N)\leq 1$, and $N(\theta)=rD(\theta)$.
			
			\item Write $D(u)=d_1u+d_0$, $N(u)=n_1u+n_0$, where $d_0,d_1,n_0,n_1\in F$, and define $\delta\coloneqq n_1d_0-n_0d_1\in F$. Then $\Nm_{L/F}(D(\theta))=-\delta\ne0$.
			
			\item The linear forms $N,D\in F[u]$ have no common zero, and the isomorphism $\Phi$ of Lemma~\ref{lem:fiberiso} is given by $\Phi=N/D$.
			
			\item Let $E\subset\PP^1_{p,K}\times_K\PP^1_{u,K}$ be the zero locus of the homogenization $D^{\mathrm h}$ of $D$. Then $E$ has bidegree $(1,1)$, its restriction to the generic fiber $\PP^1_{u,F}$ is the $F$-point $\Phi^{-1}(\infty)$, and $C$ and $E$ have no common irreducible component.
		\end{enumerate}
	\end{proposition}
	
	\begin{proof}
		(1) Since $\alpha\ne0$ by Lemma~\ref{lem:poledegrees}, we have $\deg_u(D)=1$. Moreover, \[
		\begin{aligned}
			N(u)&=(\alpha u+b\alpha-\beta)(\alpha u^2+\beta u+\gamma)-\alpha^2(u^3+bu^2+cu+e)\\
			&=(\alpha\gamma+b\alpha\beta-\beta^2-c\alpha^2)u+b\alpha\gamma-\beta\gamma-e\alpha^2,
		\end{aligned}
		\]
		so that $\deg_u N\le1$. Since $j(\theta)=0$ and $R(\theta)=r$, we have
		\[N(\theta)=D(\theta)R(\theta)-\alpha^2j(\theta)=rD(\theta).
		\]
		
		(2) We have $d_1=\alpha\ne0$. Let $u_0=-d_0/\alpha$ be the zero of $D$. The definition of $N$ gives $N(u_0)=-\alpha^2j(u_0)$, and on the other hand
		\[
		N(u_0)=n_1u_0+n_0=-\frac{n_1d_0-n_0\alpha}{\alpha}=-\frac{\delta}{\alpha}.
		\]
		Hence $\delta=\alpha^3j(u_0)$. If $\theta_1,\theta_2,\theta_3$ are the Galois conjugates of $\theta$, then
		\[\Nm_{L/F}(D(\theta))=\prod_{i=1}^3  \alpha(\theta_i-u_0)=-\alpha^3j(u_0)=-\delta.\]
		Finally, since $D$ has degree $1$ whereas the minimal
		polynomial $j$ of $\theta$ has degree $3$, we have $D(\theta)\ne0$ and hence $\delta\ne0$.
		
		(3) By (1) and (2), $N$ and $D$ define an isomorphism
		$\frac{N}{D}\colon\PP^1_{u,F}\xrightarrow{\sim}\PP^1_{\lambda,F}$. Moreover, $		N(\theta)/D(\theta)=r=h(\theta)$, so the restriction of $N/D$ to $C_F=\Spec L$ agrees with $h$. By the uniqueness part of Lemma~\ref{lem:fiberiso}, this implies $\Phi=N/D$.
		
		(4) Since $D(u)=\alpha(p)u+b\alpha(p)-\beta(p)$, Lemma~\ref{lem:poledegrees} shows that $\deg_u(D)=1$ and that $\deg_p(D)\leq 1$. Since $\deg\alpha=1$, the term $\alpha(p)u$ contains a nonzero $pu$-term, so $\deg_p(D)=1$. Hence $D^{\mathrm h}$ is a global section of $\OO(1,1)$, and its zero divisor $E$ has bidegree $(1,1)$.
		
		Over $F=K(p)$, the equation $D=0$ defines $\Phi^{-1}(\infty)$ by (3), so that $E_F=\Phi^{-1}(\infty)$. Finally, $C$ is irreducible of bidegree $(1,3)$, and therefore cannot be an irreducible component of the divisor $E$ of bidegree $(1,1)$.
	\end{proof}
	
	\begin{remark}
		An explicit computation shows that the divisor $E$ is in fact irreducible, that is, that it is not the union of a horizontal component and a vertical component. Since we will not need this stronger statement below, we omit the proof.
	\end{remark}
	
	\subsection{The pulled-back double cover}\label{sec:twist}
	
	\begin{lemma}\label{lem:signnorm}
		We have
		\[
		K(Y)=F(\lambda)\left(\sqrt{-q\Nm_{L(\lambda)/F(\lambda)}(\lambda-r)}\right).
		\]
	\end{lemma}
	
	\begin{proof}
		By Proposition~\ref{prop:surfacepencil}, the extension
		$K(Y)/F(\lambda)$ is obtained by adjoining a square root of
		$\Disc_U(H_p(U)+\lambda)$.
		
		Let $f(U)$ be a polynomial of degree $n$ with leading coefficient $a$, and write $f'(U)=na\prod_{i=1}^{n-1}(U-c_i)$ in a splitting field. Then
		\[
		\Disc_U(f(U)+\lambda)=(-1)^{n(n-1)/2}n^na^{n-1}\prod_{i=1}^{n-1}(f(c_i)+\lambda);
		\]
		see for example \cite[Lemma 2.8(ii)]{BSEF23}. We apply this to $f(U)=H_p(U)$. Since $H_p'(U)=6qU^2j(U)$, $H_p(0)=0$ and $H_p(\theta_i)=-h(\theta_i)$ for $i=1,2,3$, we obtain
		\[\Disc_U(H_p(U)+\lambda)=-6^6q^5\lambda^2\prod_{i=1}^3(\lambda-h(\theta_i))=-6^6q^5\lambda^2
		\Nm_{L(\lambda)/F(\lambda)}(\lambda-r).\]
		As $6^6q^4\lambda^2=(6^3q^2\lambda)^2$ is a square in $F(\lambda)$, the discriminant has the same square class as $-q\Nm_{L(\lambda)/F(\lambda)}(\lambda-r)$, and the result follows.
	\end{proof}
	
	\begin{proposition}\label{prop:exactcover}
		Under the isomorphism $\Phi^*\colon F(\lambda)\xrightarrow{\sim}F(u)$,
		the quadratic extension $K(Y)/F(\lambda)$ becomes $F(u)(\sqrt{qj(u)D(u)})/F(u)$.
	\end{proposition}
	
	\begin{proof}
		Write
		\[
		D(u)=d_1u+d_0,\qquad
		N(u)=n_1u+n_0,\qquad
		\delta=n_1d_0-n_0d_1,
		\]
		as in Proposition~\ref{prop:fourthdivisor}. By Proposition~\ref{prop:fourthdivisor}(1)--(2), we have $D(u)\ne0$ and $\delta\ne0$.
		
		For every conjugate
		$\theta_i$ of $\theta$, we have
		\[
		N(u)D(\theta_i)-N(\theta_i)D(u)
		=
		\delta(u-\theta_i).
		\]
		Since $N(\theta_i)/D(\theta_i)=h(\theta_i)$, it follows that
		\[
		\frac{N(u)}{D(u)}-h(\theta_i)
		=
		\frac{\delta(u-\theta_i)}
		{D(u)D(\theta_i)}.
		\]
		Multiplying these identities for $i=1,2,3$ and applying Proposition~\ref{prop:fourthdivisor}(2) gives
		\[\Nm_{L(u)/F(u)}
		\left(\frac{N(u)}{D(u)}-r\right)=
		\frac{\delta^3\prod_{i=1}^3(u-\theta_i)}
		{D(u)^3\prod_{i=1}^3D(\theta_i)}=
		-\frac{\delta^2j(u)}{D(u)^3}.
		\]
		By Lemma~\ref{lem:signnorm}, we have
		\[K(Y)=F(\lambda)\left(\sqrt{-q\Nm_{L(\lambda)/F(\lambda)}(\lambda-r)}\right).
		\]
		Since $\Phi^*(\lambda)=N(u)/D(u)$, the pulled-back quadratic extension is
		\[F(u)\left(\sqrt{q\delta^2j(u)D(u)^{-3}}\right)=F(u)\left(\sqrt{qj(u)D(u)}\right).\qedhere\]
	\end{proof}

	Let $j^{\mathrm h}$ and $D^{\mathrm h}$ be the bihomogenizations of
	$j(u)$ and $D(u)$, respectively, and let
	\[
	\pi\colon Z\longrightarrow
	\PP^1_{p,K}\times_K\PP^1_{u,K}
	\]
	be the normalization of the double cover of $\PP^1_{p,K}\times_K\PP^1_{u,K}$ defined by the section
	\[
	qj^{\mathrm h}D^{\mathrm h}
	\in
	H^0(
	\PP^1_{p,K}\times_K\PP^1_{u,K},
	\OO(2,4)
	)=H^0(
	\PP^1_{p,K}\times_K\PP^1_{u,K},
	\OO(1,2)^{\otimes 2}
	).
	\]
	The branch divisor of $\pi$ is $C+E$; see for example 
	\cite[Remark~3.14(a), Lemma~3.15(b),(d)]{EV92}.
	By Proposition~\ref{prop:exactcover}, $Z$ is birational to $Y$.
	
	\begin{corollary}\label{cor:conic-fibration}
		Consider the composite
		\[
		\rho\colon
		Z\xlongrightarrow{\pi}
		\PP^1_{p,K}\times_K\PP^1_{u,K}
		\xlongrightarrow{\operatorname{pr}_u}
		\PP^1_{u,K}.
		\]
		The generic fiber of $\rho$ is a smooth conic over $K(u)$ with a
		$K(u)$-rational point. In particular, $Z$ is $K$-rational.
	\end{corollary}
	
	\begin{proof}
		By Lemma~\ref{lem:ramificationcurve} and Proposition~\ref{prop:fourthdivisor}(4), the curves $C$ and $E$ have bidegrees $(1,3)$ and $(1,1)$, respectively, and have no common irreducible component. Moreover, by Proposition~\ref{prop:exactcover} and the definition of $Z$, the branch divisor of $\pi$ is $C+E$. It follows that, over the generic point of $\PP^1_{u,K}$, each of $C$ and $E$ meets $\PP^1_{p,K(u)}$ in a single $K(u)$-rational point, and these two points are distinct. Hence the generic fiber of $\rho$ is a degree-$2$ cover of $\PP^1_{p,K(u)}$ branched at two distinct rational points, and is therefore a smooth conic. The ramification point above either branch point is $K(u)$-rational. Thus the generic conic has a $K(u)$-rational point and is therefore $K(u)$-rational. Hence $K(Z)=K(u)(t)$ for some transcendental element $t$, so $Z$ is $K$-rational.
	\end{proof}

	\subsection{Proof of Theorem~\ref{thm:quadric}}
	\label{sec:rationality}
	We now prove Theorem~\ref{thm:quadric}. Let $K=k(q)$. By Propositions~\ref{prop:surfacepencil} and \ref{prop:exactcover}, we have $k(Q)^{A_6}=K(Y)=K(Z)$. By Corollary~\ref{cor:conic-fibration}, $Z$ is $K$-rational. Therefore $Q/A_6$ is $k$-rational, as desired.

	\section{Proof of Theorem~\ref{thm:main}}
	\label{sec:descent}
	
	Let $k$ be a field such that $2,-3\in k^{\times2}$. Recall that
	\[
	W=\{(x_1,\ldots,x_6)\in k^6:\textstyle\sum_i x_i=0\},
	\qquad
	Q=\{[x]\in\PP_k(W):\textstyle\sum_i x_i^2=0\}.
	\]
	Set
	\[
	\mathcal B=\PP_k(W)/A_6,\qquad X=Q/A_6,\qquad
	M=(\PP_k(W)\times Q)/A_6,
	\]
	where $A_6$ acts diagonally on the product. The two projections give a diagram
	\[
	\begin{tikzcd}
		& M \arrow[dl] \arrow[dr] & \\
		\mathcal B && X.
	\end{tikzcd}
	\]
	
	We will prove that both projections have rational generic
	fibers over their respective fields of definition. Since $X$ is
	rational by Theorem~\ref{thm:quadric}, this will show that
	$\mathcal B\times\PP^3_k$ is rational.
	
	It is not difficult to show that the $A_6$-actions on $\PP_k(W)$ and on $Q$ are faithful, and hence generically free. Let
	\[
	E_0\coloneqq k(\PP_k(W)),\qquad F_0\coloneqq E_0^{A_6}=k(\mathcal B).
	\]
	Thus $E_0/F_0$ is a Galois extension with group $A_6$. Equivalently, $T\coloneqq \Spec E_0$ is an $A_6$-torsor over $F_0$.
	
	The generic fiber of $M\to\mathcal B$ is the twisted quadric ${}^TQ\subset \PP_{F_0}({}^TW)\simeq \PP^4_{F_0}$, and the generic fiber of $M\to X$ is a split projective
	space $\PP^4_{k(X)}$; see for example \cite[Lemma~10.1(a),(b)]{DR15}. In particular, $M$ is birationally equivalent to $X\times\PP^4_k$.
	
	\begin{lemma}\label{lem:sylow}
		The subgroup
		\[
		P=\langle\sigma,\tau\rangle\subset A_6,\qquad
		\sigma=(1234)(56),\quad \tau=(13)(56),
		\]
		is a Sylow $2$-subgroup of order $8$ and index $45$. Let $U=k^2$
		be the $P$-representation with action
		\[
		\sigma(a,b)=(-b,a),\qquad \tau(a,b)=(-a,b).
		\]
		There is a $P$-equivariant closed immersion
		$\PP_k(U)\hookrightarrow Q$ whose image is a smooth plane conic.
	\end{lemma}
	
	\begin{proof}
		We have $\sigma^4=\tau^2=1$ and
		$\tau\sigma\tau=\sigma^{-1}$. It follows that $P$ is isomorphic to the dihedral group of order $8$. Since $|A_6|=360$, this is a Sylow $2$-subgroup
		and has index $45$. It is immediate to check that $U$ is a $P$-representation.	
		
		Since $\sigma^2=(13)(24)$ is central in $P$, its fixed subspace $W^{\sigma^2}\subset W$ is $P$-invariant. It consists of the
		vectors $(a,b,a,b,c,d)\in k^6$ such that $2a+2b+c+d=0$, and hence has dimension three. We parametrize it by
		\[
		(r,s,z)\longmapsto
		(r+s,r-s,r+s,r-s,-2r+z,-2r-z).
		\]
		Its projectivization is a plane in $\PP_k(W)$, and its intersection with
		$Q$ is the conic of equation $12r^2+4s^2+2z^2=0$. The action of $P$ on $W^{\sigma^2}$ is given in these coordinates by
		\[
		\sigma(r,s,z)=(r,-s,-z),
		\qquad
		\tau(r,s,z)=(r,s,-z).
		\]
		Choose $\eta,\xi\in k^\times$ such that $\eta^2=-3$ and $\xi^2=2$. Consider the $P$-equivariant closed immersion
		\[
		\PP_k(U)\longrightarrow \PP_k(W^{\sigma^2}),
		\qquad
		[a:b]\longmapsto [a^2+b^2:\eta(a^2-b^2):2\eta\xi ab].
		\]
		Its image is contained in the plane conic $\PP(W^{\sigma^2})\cap Q$. We obtain the desired $P$-equivariant closed immersion $\PP_k(U)\hookrightarrow Q$.
	\end{proof}
	
	\begin{remark}
		Over an algebraic closure of $k$, the $P$-quadric $Q$ appears in \cite[Remark~6.6]{HT25}, where its stable linearizability was left open. The stable linearizability of this action was later proved in \cite[Theorem~4.1, Case~(1)]{CTZ26}; their argument also uses an invariant conic in $\PP(\chi_1\oplus\chi_2\oplus\chi_3)$. The construction in Lemma~\ref{lem:sylow} gives another direct proof of stable linearizability: after twisting by any $P$-torsor $S/F$, the $P$-equivariant immersion $\PP(U)\hookrightarrow Q$ yields $\PP({}^S U)\simeq\PP^1_F\hookrightarrow{}^S Q$, so every twist ${}^S Q$ has an $F$-rational point, and the conclusion follows from \cite[Theorems~1.1(a) and~10.2]{DR15}. To our knowledge, the linearizability of this $P$-action remains open.
		
		Since \cite[Theorem~4.1, Case~(1)]{CTZ26} is stated and proved only over an algebraically closed field of characteristic zero, we prove Lemma~\ref{lem:sylow} directly rather than appeal to that result. This also explains the role of the hypotheses $2,-3\in k^{\times2}$.
	\end{remark}

	\begin{proposition}\label{prop:comparison}
		We have
		\[
		\mathcal B\times\PP^3_k
		\bir M
		\bir X\times\PP^4_k.
		\]
		In particular, $\mathcal B\times\PP^3_k$ is $k$-rational.
	\end{proposition}
	
	The reduction to a Sylow $2$-subgroup in the proof below is an instance of the versality criterion for quadrics of Duncan--Reichstein \cite[Theorem~10.2]{DR15}; see also \cite[Proposition~5.3]{HT25}. A related argument, based on twisting an equivariant morphism from $\PP^1$, appears in the proof of \cite[Lemma~2.2]{Sca26}.

	\begin{proof}
		The birational equivalence $M\bir X\times\PP^4_k$ was observed above. For the other birational equivalence, set $F_1\coloneqq E_0^P$, and consider the $P$-torsor
		\[
		S\coloneqq \Spec E_0\longrightarrow \Spec F_1.
		\]
		The base change of the $A_6$-torsor
		$T\to\Spec F_0$ to $F_1$ admits $S$ as a reduction of structure group
		from $A_6$ to $P$, that is, we have an isomorphism of $A_6$-torsors over $F_1$:
		\[
		S\times^P A_6 \simeq T_{F_1}.
		\]
		Therefore
		\[
		({}^TQ)_{F_1}
		\simeq
		(S\times^P A_6)\times^{A_6}Q
		\simeq
		S\times^P Q
		=
		{}^SQ.
		\]
		Twisting the $P$-equivariant closed
		immersion of Lemma~\ref{lem:sylow} gives a closed immersion
		$
		\PP^1_{F_1}\hookrightarrow({}^TQ)_{F_1},
		$
		so $({}^TQ)(F_1)\neq\emptyset$. Since $[F_1:F_0]=[A_6:P]=45$ is odd, Springer's theorem \cite{Spr52} (see also \cite[Corollary~18.5]{EKM08}) implies that ${}^TQ(F_0)\neq\emptyset$. Hence the smooth quadric threefold ${}^TQ$ is $F_0$-rational. As ${}^TQ$ is the generic fiber of $M\to\mathcal B$, we obtain $M\bir\mathcal B\times\PP^3_k$. Finally, $X$ is $k$-rational by Theorem~\ref{thm:quadric}, and hence so is $\mathcal B\times\PP^3_k$.
	\end{proof}
	
	\begin{proof}[Proof of Theorem~\ref{thm:main}]
		View $W\setminus\{0\}\to\PP_k(W)$ as the complement of the zero
		section in the tautological line bundle. Let $U\subset\PP_k(W)$ be the free locus. Then
		$U\to U/A_6$ is an $A_6$-torsor. Since the tautological line bundle $\OO_{\PP(W)}(-1)$ is naturally $A_6$-linearized, its restriction to $U$ descends to a line bundle on $U/A_6$; see for example \cite[Theorem~3.6, Remark~3.7]{CTS07}. It follows that
		\[
		W/A_6\bir \mathcal B\times\A^1_k.
		\] 
		Let $V$ be the $6$-dimensional permutation representation of $A_6$ over $k$. Since $6$ is invertible in $k$, we have $A_6$-equivariant decomposition $V=W\oplus k$, where the action on the last summand is trivial, and hence
		\[
		V/A_6\bir W/A_6\times \A^1_k\bir\mathcal B\times\A^2_k.
		\]
		It now follows from Proposition~\ref{prop:comparison} that
		\[
		V/A_6\times\A^1_k\bir\mathcal B\times\A^3_k
		\]
		is rational, as desired.
	\end{proof}

\end{document}